\documentclass[11pt,a4paper]{amsart}
\usepackage[T1]{fontenc}
\usepackage[utf8]{inputenc}
\usepackage{lmodern}
\usepackage{amsmath,amssymb,mathtools}
\usepackage{microtype}
\usepackage[left=27mm,right=27mm,top=26mm,bottom=26mm]{geometry}
\usepackage{enumitem}
\usepackage[hidelinks,unicode,pdfencoding=auto]{hyperref}

\numberwithin{equation}{section}
\newtheorem{theorem}{Theorem}[section]
\newtheorem{lemma}[theorem]{Lemma}
\newtheorem{proposition}[theorem]{Proposition}
\newtheorem{corollary}[theorem]{Corollary}
\theoremstyle{remark}
\newtheorem{remark}[theorem]{Remark}
\newcommand{\R}{\mathbb R}
\newcommand{\N}{\mathbb N}
\newcommand{\Z}{\mathbb Z}
\newcommand{\cS}{\mathcal S}
\newcommand{\cC}{\mathcal C}
\newcommand{\cK}{\mathcal K}
\newcommand{\dd}{\,\mathrm dt}

\newcommand{\dist}{\operatorname{dist}}
\newcommand{\rad}{\operatorname{rad}}

\newcommand{\doi}[1]{\href{https://doi.org/#1}{\nolinkurl{#1}}}
\DeclarePairedDelimiter{\norm}{\lVert}{\rVert}

\setlist[enumerate]{label=\textup{(\roman*)},leftmargin=2.1em,itemsep=3pt,topsep=5pt}
\allowdisplaybreaks[2]
 
\title[The minimal period solution for even potentials]{The minimal period solution of Hamiltonian systems with even potentials}
\author[Huagui Duan]{Huagui Duan}
\thanks{This work was partially supported by NSFC (No. 12631007) and NSF of Tianjin (No. 25JCZDJC01030).}
\address{Huagui Duan, School of Mathematics, Jilin University, Changchun, Jinlin, 130012, The People's Republic of China}
\email{duanhg@nankai.edu.cn}
\author[Yan Wu]{Yan Wu}
\address{Yan Wu, School of Mathematical Sciences, Nankai University, Tianjin 300071, The People's Republic of China}
\date{October 21, 2026}

\begin{document}
\begin{abstract}
We consider the autonomous system $\ddot{q}+\nabla V(q)=0$ in an arbitrary finite-dimensional Euclidean space. The potential is assumed to be even, nonnegative and of class $C^2$, to satisfy $V(x)=o(|x|^2)$ at the origin, and the Ambrosetti--Rabinowitz superquadratic inequality outside a ball. We prove that, for every prescribed $T>0$, the system has a nonconstant solution with $T$ being its minimal positive period.
\end{abstract}
\maketitle

\section{Introduction and main result}\label{sec:introduction}

The prescribed minimal period problem asks whether every $T>0$ occurs as the least positive period of a nonconstant solution of an autonomous Hamiltonian system.

Consider the following second-order mechanical system
\begin{equation}\label{eq:system}
 \ddot{q}(t)+\nabla V(q(t))=0,\qquad q(t)\in\R^n,
\end{equation}
where $\dot{q}(t)=\frac{d}{dt}q(t)$, $V:\mathbb{R}^n\rightarrow\mathbb{R}$ is the potential energy, and $\nabla V(q)$ denotes the gradient of $V$ with respect to $q\in\mathbb{R}^n$.

Throughout this paper, the following assumptions about the potential $V$ will be used:
\begin{itemize}
\item[($\mathrm{V1}$)] \, $V\in C^1(\R^n,\R)$,\vspace{1mm}
\item[($\mathrm{V1'}$)] \, $V\in C^2(\R^n,\R)$,\vspace{1mm}
\item[($\mathrm{V2}$)] $\, \text{there exist }\mu>2,\ r_0>0\text{ such that}$ $0<\mu V(q)\le\langle\nabla V(q),q\rangle$ for $|q|\ge r_0$,\vspace{1mm}
\item[($\mathrm{V3}$)] $\, V(q)=o(|q|^2)\quad\text{as }q\to0$, i.e. $\lim_{|q|\rightarrow 0}\frac{V(q)}{|q|^2}=0$, \vspace{1mm}
\item[($\mathrm{V4}$)] $\,V(q)\ge0\quad\text{for every }q\in\R^n$,\vspace{1mm}
\item[($\mathrm{V5}$)] $\, V(-q)=V(q)\quad\text{for every }q\in\R^n$.
\end{itemize}

In 1978, Rabinowitz~\cite{Rabinowitz1978} proved the existence of nonconstant $T$-periodic solutions if the potential satisfies ($\mathrm{V1}$)-($\mathrm{V4}$). Note that a $T$-periodic solution can also be an iterate of an orbit with a shorter period, so this existence statement alone does not determine its minimal period. Rabinowitz also asked whether, under $(\mathrm{V1})$--$(\mathrm{V4})$, the system admits a nonconstant solution with any prescribed minimal period.

The system \eqref{eq:system} can be written as
\[
 \begin{pmatrix}\dot p(t)\\ \dot q(t)\end{pmatrix}
 =\begin{pmatrix}-\nabla V(q(t))\\p(t)\end{pmatrix}
 =J\nabla H(x),
\]
where
\[
 J=\begin{pmatrix}0&-I\\I&0\end{pmatrix},\qquad
 x=(p,q)\in\R^{2n},\qquad
 H(p,q)=\frac12|p|^2+V(q),
\]
and $I$ is the $n\times n$ identity matrix. Here
$\nabla H(x)=(p,\nabla V(q))$. Thus \eqref{eq:system} is equivalent to the
first-order Hamiltonian system
\begin{equation}\label{1-order-eq}
 \dot x=J\nabla H(x),\qquad x\in\R^{2n}.
\end{equation}

For convex first-order Hamiltonian systems, Ekeland and Hofer~\cite{EkelandHofer1985} obtained prescribed minimal-period solutions by relating the variational characterization of a critical point to its index and local geometry. For the second-order problem with the potential $V$ satisfying ($\mathrm{V1'}$) and ($\mathrm{V2}$)-($\mathrm{V5}$), Long~\cite{Long1993} worked in a space of time-even, half-antiperiodic loops and proved that a suitable symmetric mountain pass solution has minimal period either $T$ or $T/3$ (also cf. Theorem 2.1 of \cite{Long2021}). The possibility of a triple iterate is the specific issue addressed here.

In addition, there are also prescribed minimal-period results under additional radial hypotheses. Xiao~\cite{Xiao2010} treated even potentials satisfying a global radial Hessian inequality. Xiao and Zhu~\cite{XiaoZhu2025} obtained related conclusions under a weak Nehari condition; the global differential inequality used there should be distinguished from an Ambrosetti--Rabinowitz condition imposed only outside a ball. Further Nehari-type results appear in Xiao and Chen~\cite{XiaoChen2024}. Discrete variational approximation for even systems was developed by Kuang, Chen and Guo~\cite{KuangChenGuo2022}, and related results for nonsymmetric potentials were studied by Kuang and Guo~\cite{KuangGuo2024}.

For $T>0$, define
\begin{align}
 E_T&=W^{1,2}(\R/T\Z,\R^n),\\
 \cS_T&=\{u\in E_T:u(-t)=u(t),\ u(t+T/2)=-u(t)\},\label{eq:symmetric-space}\\
 \psi_T(u)&=\int_0^T\left(\frac12|u'(t)|^2-V(u(t))\right)\dd.\label{eq:action}
\end{align}
We define the path class with endpoint of negative functional value and its minimax value:
\begin{equation}\label{eq:minimax}
 \Gamma_T=\{\gamma\in C([0,1],\cS_T):\gamma(0)=0,\
                    \psi_T(\gamma(1))<0\},\quad
 c_T=\inf_{\gamma\in\Gamma_T}\max_{s\in[0,1]}\psi_T(\gamma(s)).
\end{equation}
Here continuity of paths is always understood in the $W^{1,2}$ topology. Whenever $c_T>0$, let $\cC_T$ denote the path component containing zero in
\begin{equation}\label{eq:zero-component}
 \{u\in\cS_T:\psi_T(u)<c_T\}.
\end{equation}

Motivated by these above results, especially the first problem proposed by Long in Section~5 of \cite{Long2021}, we prove the following theorem.

\enlargethispage{\baselineskip}
\begin{theorem}\label{thm:main}
Let $n\ge1$ be an integer and let $V:\R^n\to\R$ satisfy $(\mathrm{V1'})$ and $(\mathrm{V2})$--$(\mathrm{V5})$. Then, for every $T>0$, the system \eqref{eq:system} possesses a nonconstant classical solution
\[
 q\in C^2(\R,\R^n)\cap\cS_T
\]
whose minimal positive period is exactly $T$. More precisely, the solution satisfies
\begin{equation}\label{eq:main-selection}
 \psi_T(q)=c_T>0,\qquad
 D(\psi_T|_{\cS_T})(q)=0,\qquad
 q\in\overline{\cC_T}.
\end{equation}
\end{theorem}


Our proof uses global information carried by a mountain pass selection within the variational framework of Ambrosetti and Rabinowitz~\cite{AmbrosettiRabinowitz1973}. The local geometry and location of mountain pass critical points have been studied by Hofer~\cite{Hofer1985} and Ghoussoub and Preiss~\cite{GhoussoubPreiss1989}; component-based selection also occurs in the study of partially convex reversible systems by Zhou~\cite{Zhou2025}. We give a self-contained version for a mountain pass class whose terminal point may be any point of negative action. This version selects a critical point in the closure of the component containing zero in the strict sublevel below the mountain pass value.

\subsection*{Sketch of proof of Theorem 1.1}
Fix $T>0$. The symmetric space $\cS_T$ excludes nonzero constant loops.
The assumptions at the origin and at infinity give mountain pass geometry
and the Palais--Smale condition for $\psi_T|_{\cS_T}$
(cf. Lemmas~\ref{lem:geometry} and~\ref{lem:ps}). A deformation argument for the
path class \eqref{eq:minimax}, whose negative action endpoint is free,
selects a critical point $q$ with
\[
 \psi_T(q)=c_T>0,\qquad q\in\overline{\cC_T}
\]
(cf. Corollary~\ref{cor:selected-solution}). The component $\cC_T$ contains no
point of negative action: a path from zero to such a point inside
$\{\psi_T<c_T\}$ would contradict the definition of $c_T$.

Let $\tau$ be the minimal period of $q$. Its reflection symmetry and
half-antiperiodicity imply $T=k\tau$ with $k$ odd
(cf. Lemma~\ref{lem:periods}). If $\tau<T$, write $k=2m+1\ge3$ and restrict
the orbit to $[0,L]$, where $L=T/4$. The reflection isomorphism gives
$\psi_T=4I_L$. Cutoff velocity fields yield an isotropic vector $h$ outside
the radical of $D^2I_L(q)$ and a negative definite subspace of dimension
$m$ (cf. Lemma~\ref{lem:cutoff}).

On this quarter-period interval, we construct paths by deleting part of
the motion and inserting a waiting segment at a turning point. Their
action losses are explicit integrals of $-|q'|^2$, and their terminal
actions are negative. For $k=3$, two such paths have opposite first-order
tangents $h$ and $-2h$. A local coordinate in which the functional has the
form $I_L(q)+g(y)-\xi^2$ connects each nearby point of smaller action to
one of these paths. For $k\ge5$, the negative subspace contains a
negative two-plane, which connects the two local exit directions, so one
surgery path suffices. Therefore, Lemma~\ref{lem:local-exits} gives the
strict sublevel exit property in Proposition~\ref{prop:exit}.

Since $q\in\overline{\cC_T}$, this exit property applies to a point of
$\cC_T$ sufficiently close to $q$. Joining zero to that point within
$\cC_T$, and then following its exit path to negative action, produces an
admissible path whose compact image has maximum strictly below $c_T$.
This contradiction excludes every odd iterate $k\ge3$ and proves
$\tau=T$.

Section~\ref{sec:preliminaries} establishes the variational framework,
compactness, boundary selection, local sublevel geometry, and the required
negative directions. Section~\ref{sec:mainproof} proves
Theorem~\ref{thm:main} using the orbit-surgery exclusion. The conclusion
concerns an appropriately selected symmetric solution. It does not assert
that every critical point obtained from every fixed negative endpoint has
minimal period $T$.

\section{Some preliminaries}\label{sec:preliminaries}

\subsection{The variational framework}

We use the notation $H^1=W^{1,2}$. More explicitly, $E_T$ consists of
functions $u\in L^2(0,T;\R^n)$ whose weak derivative $u'$ belongs to
$L^2(0,T;\R^n)$ and satisfy $u(0)=u(T)$. The weak
derivative is characterized by
\[
 \int_0^T\langle u,\phi'\rangle\dd
 =-\int_0^T\langle u',\phi\rangle\dd,
 \qquad\forall\phi\in C_c^\infty((0,T),\R^n).
\]
The inner product
\[
 (u,v)_{H^1}=\int_0^T
       \bigl(\langle u,v\rangle+\langle u',v'\rangle\bigr)\dd
\]
makes $E_T$ to be a Hilbert space. This $H^1$-norm satisfies
$\|u\|_{H^1}^2=\|u\|_2^2+\|u'\|_2^2$.

On a bounded interval, the fundamental theorem for Sobolev functions
states that $w\in W^{1,p}(a,b)$, $1\le p\le\infty$, has an absolutely
continuous representative such that
\[
 w(t)-w(s)=\int_s^t w'(r)\,\mathrm dr,
 \qquad\forall s,t\in[a,b];
\]
see \cite[Theorem~8.2]{Brezis2011}. In fact, the
primitive of the weak derivative has that derivative, and its difference
from $w$ has zero distributional derivative and is constant. We always
use this representative and extend every loop periodically to $\R$.
The assumptions also imply that $V(0)=0$ and $\nabla V(0)=0$.

For $c_T>0$, put $A_T=\{u\in\cS_T:\psi_T(u)<c_T\}$. Then
\[
 \cC_T=\{u\in A_T:\exists\eta\in C([0,1],A_T),\
                   \eta(0)=0,\ \eta(1)=u\}.
\]
Thus $q\in\overline{\cC_T}$ means that every $H^1$-ball about $q$ in
$\cS_T$ contains a point of $\cC_T$.

\begin{lemma}\label{lem:geometry}
Assume that $V$ satisfies $(\mathrm{V1'})$ and
$(\mathrm{V2})$--$(\mathrm{V5})$. Then $\cS_T$ is a closed Hilbert
subspace of $E_T$, and
\[
 \norm{u}_T:=\norm{u'}_{L^2(0,T)}
\]
is a Hilbert norm equivalent to the $H^1$-norm on $\cS_T$. The
restriction $\psi_T|_{\cS_T}$ is of class $C^2$, the class $\Gamma_T$
is nonempty, and $0<c_T<\infty$.

If $L=T/4$, the restriction to $[0,L]$ is a bounded linear isomorphism
\begin{equation}\label{eq:quarter-space}
 \cS_T\simeq X_L:=\{u\in H^1([0,L],\R^n):u(L)=0\}.
\end{equation}

Let $R_Lu=u|_{[0,L]}$, we have
\begin{equation}\label{eq:quarter-action}
 \psi_T(u)=4I_L(R_Lu),\qquad
 I_L(w)=\int_0^L\left(\frac12|w'(t)|^2-V(w(t))\right)\dd.
\end{equation}
Here $w\in X_L$. We identify a symmetric loop with its restriction
through \eqref{eq:quarter-space}; thus $\psi_T(u)=4I_L(u)$ will also
be used when the domain of the argument is clear.

\end{lemma}

\begin{proof}
On $E_T$, define the reflection and half-period translation by
\[
 (\mathcal Ru)(t)=u(-t),\qquad
 (\mathcal Au)(t)=u(t+T/2).
\]
These two maps are linear isometries for the $H^1$-norm. Hence
$\cS_T=\ker(\mathcal R-\mathrm{Id})\cap
\ker(\mathcal A+\mathrm{Id})$ is closed. For $u\in\cS_T$, its two
symmetries give
\[
 u(T/4)=u(-T/4)=-u(T/4),
\]
so $u(T/4)=0$. Then the fundamental theorem gives
\[
 u(t)=\int_{T/4}^t u'(r)\,\mathrm dr,\qquad\forall 0\le t\le T.
\]
Let $I_t$ be the interval with endpoints $T/4$ and $t$, the
Cauchy--Schwarz inequality yields
\[
 |u(t)|\le |t-T/4|^{1/2}
             \left(\int_{I_t}|u'(r)|^2\,\mathrm dr\right)^{1/2}
          \le\sqrt T\,\|u'\|_{L^2(0,T)}.
\]
Taking the supremum and then integrating its square proves
\begin{equation}\label{eq:sobolev-estimates}
 \|u\|_\infty\le\sqrt T\,\|u'\|_2,\qquad
 \|u\|_2^2\le T\|u\|_\infty^2\le T^2\|u'\|_2^2.
\end{equation}
In particular,
\[
 \|u'\|_2\le\|u\|_{H^1}\le\sqrt{1+T^2}\,\|u'\|_2.
\]
Thus this proves the norm equivalence.

For $w\in X_L$, we define its extension on $[0,4L]$ by
\[
 (\mathcal E_Lw)(t)=
 \begin{cases}
  w(t),&0\le t\le L,\\
  -w(2L-t),&L\le t\le2L,\\
  -w(t-2L),&2L\le t\le3L,\\
  w(4L-t),&3L\le t\le4L.
 \end{cases}
\]
The adjacent values match because $w(L)=0$; the values at $0$ and $4L$
also match. The extension is absolutely continuous on each piece and,
by matching, on the whole period. Its piecewise derivative belongs to
$L^2$. So It belongs to $H^1$. The formulas imply
$\mathcal E_Lw(-t)=\mathcal E_Lw(t)$ and
$\mathcal E_Lw(t+2L)=-\mathcal E_Lw(t)$ after periodic continuation.
Conversely these identities uniquely determine every quarter from the first.
Furthermore,
\[
 \|\mathcal E_Lw\|_{H^1(0,4L)}^2=4\|w\|_{H^1(0,L)}^2.
\]
Thus the extension and restriction are bounded inverse linear maps.

On each quarter, change variables to $r\in[0,L]$. Reflection or a sign
change leaves $|w'(r)|^2$ invariant, and $(\mathrm{V5})$ gives
$V(-w(r))=V(w(r))$. Hence, for $j=0,1,2,3$,
\[
 \int_{jL}^{(j+1)L}
 \left(\frac12|(\mathcal E_Lw)'(t)|^2-V(\mathcal E_Lw(t))\right)\dd
 =\int_0^L\left(\frac12|w'(r)|^2-V(w(r))\right)\,\mathrm dr.
\]
Summing the four equal integrals gets
$\psi_T(\mathcal E_Lw)=4I_L(w)$, as claimed in \eqref{eq:quarter-action}.

The continuous embedding $H^1([0,T])\hookrightarrow C^0([0,T])$ places the ranges of all loops in a bounded $H^1$ neighborhood in one compact ball of $\R^n$. Taylor's formula, together with uniform continuity of $D^2V$ on that ball, gives continuous Fr\'echet derivatives
\begin{align}
 D\psi_T(u)[v]
 &=\int_0^T\bigl(\langle u',v'\rangle-\langle\nabla V(u),v\rangle\bigr)\dd,
 \label{eq:first-variation}\\
 D^2\psi_T(u)[v,w]
 &=\int_0^T\bigl(\langle v',w'\rangle
                -\langle D^2V(u)v,w\rangle\bigr)\dd.
 \label{eq:second-variation}
\end{align}

It follows from $(\mathrm{V2})$ that the potential is positive on the sphere $|x|=r_0$.
Hence $b_0=\min_{|x|=r_0}V(x)>0$. For a unit vector $e$ and $r\ge r_0$,
\[
 \frac{\mathrm d}{\mathrm dr}\bigl(r^{-\mu}V(re)\bigr)
 =\frac{\langle\nabla V(re),re\rangle-\mu V(re)}{r^{\mu+1}}\ge0.
\]
Integrating from $r_0$ to $r$ yields
$V(re)\ge r^\mu r_0^{-\mu}V(r_0e)\ge b r^\mu$, where
$b=b_0r_0^{-\mu}$. On $|x|\le r_0$, the nonnegativity gives
$V(x)\ge0\ge b|x|^\mu-b_0$. Thus,
\begin{equation}\label{eq:lower-growth}
 V(x)\ge b|x|^\mu-b_0,\qquad\forall x\in\R^n.
\end{equation}

Choose $\varepsilon=1/(4T^2)$. By $(\mathrm{V3})$, there is $\delta>0$
such that $V(x)\le\varepsilon|x|^2$ for $|x|\le\delta$. Choose
$\rho>0$ with $\sqrt T\rho\le\delta$. If $\|u\|_T\le\rho$,
\eqref{eq:sobolev-estimates} gives $\|u\|_\infty\le\delta$, and hence
\begin{equation}\label{eq:small-ball}
 \begin{aligned}
 \psi_T(u)\ge\tfrac12\|u'\|_2^2-\varepsilon\|u\|_2^2
 \ge\left(\tfrac12-\varepsilon T^2\right)\|u\|_T^2
   =\tfrac14\|u\|_T^2.
 \end{aligned}
\end{equation}

For every $\gamma\in\Gamma_T$, its negative action endpoint satisfies
$\|\gamma(1)\|_T>\rho$, while $\|\gamma(0)\|_T=0$. Then the continuity and the intermediate value theorem give some $s_*\in(0,1)$
with $\|\gamma(s_*)\|_T=\rho$. So it yields $\max_s\psi_T(\gamma(s))\ge\rho^2/4$.

For a fixed nonzero $u\in\cS_T$, set
$A=\int_0^T|u(t)|^\mu\dd>0$. The estimate \eqref{eq:lower-growth} gives
\[
 \psi_T(ru)\le\tfrac12r^2\|u\|_T^2-br^\mu A+b_0T\equiv\eta(r)
 \longrightarrow-\infty\qquad(r\to\infty),
\]
due to $\lim_{r\to\infty}\eta(r)r^{-\mu}=-bA<0$. A nonzero vector times $\cos(2\pi t/T)$ supplies such a loop.
Choose $R>0$ with $\psi_T(Ru)<0$ and set $\gamma(s)=sRu$.
Then $\gamma\in\Gamma_T$, and $\psi_T\circ\gamma$ has a finite maximum
on $[0,1]$. Together with \eqref{eq:small-ball}, this proves
$\Gamma_T\ne\emptyset$ and $\rho^2/4\le c_T<\infty$.
\end{proof}

\subsection{Compactness, critical points and primitive periods}

\begin{lemma}\label{lem:ps}
Assume that $V$ satisfies ($\mathrm{V1'}$) and ($\mathrm{V2}$)-($\mathrm{V5}$). Then the functional $\psi_T|_{\cS_T}$ satisfies the Palais--Smale condition, that is, every sequence $(u_j)$ for which $(\psi_T(u_j))$ is bounded and $\norm{D(\psi_T|_{\cS_T})(u_j)}\to0$ has a strongly convergent subsequence in $H^1$. Every positive action critical point is a nonconstant classical solution of \eqref{eq:system}, and belongs to $C^3(\R,\R^n)$.
\end{lemma}

\begin{proof}
Set $F(x)=\langle\nabla V(x),x\rangle-\mu V(x)$ and
\[
 M=\max\{0,-\min_{|x|\le r_0}F(x)\}.
\]
The minimum exists by the continuity of $F$. It follows from $(\mathrm{V2})$ that
$F\ge0$ outside the ball, so $F\ge-M$ on $\R^n$.
By \eqref{eq:first-variation} with $v=u$, we obtain
\begin{equation}\label{eq:ps-bound}
 \begin{aligned}
 \mu\psi_T(u)-D\psi_T(u)[u]
 &=\left(\tfrac\mu2-1\right)\|u'\|_2^2
                  +\int_0^T F(u(t))\dd\\
 &\ge\left(\tfrac\mu2-1\right)\|u\|_T^2-MT,
       \qquad\forall u\in\cS_T.
 \end{aligned}
\end{equation}
For a P.S. sequence, choose $C\ge0$ with
$|\psi_T(u_j)|\le C$ and put
\[
 \epsilon_j=\|D(\psi_T|_{\cS_T})(u_j)\|_{(\cS_T)^*}\longrightarrow0,
\]
where the dual norm corresponds to $\|\cdot\|_T$. Let $A=\mu/2-1>0$, then
\[
 A\|u_j\|_T^2\le\mu C+MT+\epsilon_j\|u_j\|_T.
\]
If $\|u_j\|_T\to\infty$, then, after being divided by $\|u_j\|_T^2$,
the above inequality gives $A\le0$. Thus $(u_j)$ is bounded, and the norm
equivalence gives boundedness in $H^1$.

The estimates \eqref{eq:sobolev-estimates} and
\[
 |u_j(t)-u_j(s)|\le\norm{u_j'}_2|t-s|^{1/2}
\]
give the uniform boundedness and equicontinuity on a period interval. Passing to a subsequence, $u_j\rightharpoonup u$ in $H^1$ and $u_j\to u$ uniformly, which implies that $u\in\cS_T$.

Now we have
\[
 \int_0^T\langle u_j',u_j'-u'\rangle\dd
 =D\psi_T(u_j)(u_j-u)
  +\int_0^T\langle\nabla V(u_j),u_j-u\rangle\dd\longrightarrow0,
\]
where the first term tends to zero by the boundedness and the P.S. hypothesis; the second one is due to the uniform convergence and boundedness of $\nabla V$. Subtracting the term that tends to zero by weak convergence yields
\[
 \norm{u_j'-u'}_2^2\longrightarrow0.
\]
Then the norm equivalence proves the strong convergence in $H^1$.

Let $q$ be a critical point and $L=T/4$. The restriction and \eqref{eq:quarter-action} imply
$DI_L(q)[v]=0$, $\forall v\in X_L$. For a test function $\phi\in C_c^\infty((0,L),\R^n)$, extending it by zero to the
endpoints gives $\phi\in X_L$, and $q+\epsilon\phi\in X_L$ for every
$\epsilon\in\R$. Thus
\[
 0=\left.\frac{\mathrm d}{\mathrm d\epsilon}
                       I_L(q+\epsilon\phi)\right|_{\epsilon=0}
   =\int_0^L\bigl(\langle q',\phi'\rangle
                         -\langle\nabla V(q),\phi\rangle\bigr)\dd.
\]
This yields $\ddot{q}=-\nabla V(q)$ on $(0,L)$. Since $q$ and $\nabla V(q)$ are continuous, this proves
$q\in C^2([0,L],\R^n)$.

For every $v\in X_L$, noting that $v(L)=0$, by the integration by parts, we have
\[
 \begin{aligned}
 0=DI_L(q)[v]
 =[\langle q'(t),v(t)\rangle]_0^L
       -\int_0^L\langle\ddot{q}+\nabla V(q),v\rangle\dd=-\langle q'(0),v(0)\rangle.
 \end{aligned}
\]
For any $\zeta\in\R^n$, the variation
$v(t)=(1-t/L)\zeta\in X_L$ and $v(0)=\zeta$. Taking
$\zeta=q'(0)$ gives $q'(0)=0$. So we have
\begin{equation}\label{eq:mixed-boundary}
 q'(0)=0,\qquad q(L)=0.
\end{equation}

Even reflection at $0$ matches the first derivative because $q'(0)=0$.
At $L$, odd reflection matches the value because $q(L)=0$, matches the first
derivative automatically, and matches the second derivative because
$\ddot{q}(L)=-\nabla V(0)=0$. Since $\nabla V$ is odd, both reflections preserve the same equation.
They give a globally $C^2$ solution.
As $V\in C^2$, the composition $\nabla V(q(t))$ is $C^1$, and
\[
 \frac{\mathrm d}{\mathrm dt}\ddot{q}(t)
       =-D^2V(q(t))q'(t)
\]
is continuous. Thus the claimed $C^3$ regularity of $q$ is proved.
\end{proof}

\begin{lemma}\label{lem:periods}
If $q:\R\to\R^n$ is nonconstant, continuous and periodic, its set of periods is $\tau\Z$ for some $\tau>0$. Moreover, if $q\in\cS_T$, then $T=k\tau$ for a positive odd integer $k$, and $q\in\cS_\tau$.

\end{lemma}

\begin{proof}
The period set
\[
 G=\{h\in\R:q(t+h)=q(t),\ \forall t\in\R\}
\]
is an additive subgroup of $\R$. If $g_j\in G$ and
$g_j\to g$, the continuity gives $q(t+g)=\lim_jq(t+g_j)=q(t)$, $\forall t$. So It is closed.

If the positive period sequence $h_j\to 0$, then for any $s\in\R$, choose
$m_j=\lfloor s/h_j\rfloor$, it holds
\[
 0\le s-m_jh_j<h_j,
\]
so $m_jh_j\to s$. Since each integer multiple of a period is a period,
\[
 q(t+s)=\lim_{j\to\infty}q(t+m_jh_j)=q(t),
 \qquad\forall t,s\in\R,
\]
which makes $q$ to be constant. Therefore
$\tau=\inf(G\cap(0,\infty))>0$.

Choose $g_j\in G$ with $\tau\le g_j<\tau+1/j$, the closedness of $G$ gives
$\tau\in G$. For any $g\in G$, write
\[
 g=p\tau+r,\qquad p=\lfloor g/\tau\rfloor\in\Z,
 \qquad 0\le r<\tau.
\]
The subgroup property implies $r=g-p\tau\in G$. A positive $r$ would contradict
the definition of $\tau$, so $r=0$. Hence $G=\tau\Z$.
For $q\in\cS_T$, one has $T\in G$ and therefore $T=k\tau$ with
$k\in\N$.

If $k$ were even, $T/2$ would be a period. Then half-antiperiodicity give $q(t)=-q(t)$, $\forall t$, contradiction! Thus $k=2j+1$ for an integer $j\ge0$. Since $T/2=j\tau+\tau/2$, we have $q(t+\tau/2)=-q(t)$. Hence $q\in\cS_\tau$.
\end{proof}

\subsection{A boundary-selected mountain pass point}

We next isolate the precise selection property used in the proof. For a subset $A$ of a Hilbert space, $\overline A$ denotes its closure.

\begin{lemma}\label{lem:selection}
Let $H$ be a real Hilbert space, let $f\in C^2(H,\R)$ with $f(0)=0$, and define
\[
 \Gamma=\{\gamma\in C([0,1],H):\gamma(0)=0,\ f(\gamma(1))<0\},
 \qquad c=\inf_{\gamma\in\Gamma}\max_{s\in[0,1]}f(\gamma(s)).
\]
Assume $\Gamma\ne\emptyset$, $0<c<\infty$, and $f$ satisfies the P.S. condition at $c$: every sequence $x_j$ with $f(x_j)\to c$ and $\norm{Df(x_j)}\to0$ has a strongly convergent subsequence. Let $\cC$ be the path component of $\{f<c\}$ containing zero, and set
\[
 \cK_c=\{x\in H:f(x)=c,\ Df(x)=0\}.
\]
Then $\cK_c$ is nonempty and compact, and
\begin{equation}\label{eq:selection}
 \cK_c\cap\overline{\cC}\ne\emptyset.
\end{equation}

\end{lemma}

\begin{proof}
We identify the derivative with the gradient through
\[
 Df(x)[v]=(\nabla f(x),v)_H,\qquad
 \|Df(x)\|_{H^*}=\|\nabla f(x)\|_H.
\]
Thus the P.S. condition is unchanged when using
$\nabla f$. The gradient is $C^1$, since $f$ is $C^2$ and the Riesz
identification is a bounded linear map. Its derivative is bounded on a
sufficiently small ball at each point. So integration along line segments
shows that $\nabla f$ is locally Lipschitz.

\medskip

{\bf Step 1: $ \cK_c\ne\emptyset$ and  $\cK_c$ is compact.}

\medskip

Suppose first that $\cK_c=\emptyset$. The P.S. condition gives
$\eta,\sigma>0$ such that
\begin{equation}\label{eq:empty-gradient}
 \|\nabla f(x)\|\ge\sigma
 \qquad\text{when }|f(x)-c|\le2\eta.
\end{equation}
Otherwise, for every $\eta,\sigma>0$, there are $x_j\in H$
with $|f(x_j)-c|\le1/j$ and $\|\nabla f(x_j)\|<1/j$.
A strongly convergent P.S. subsequence would have a limit $x_*$
with $f(x_*)=c$ and $\nabla f(x_*)=0$, contradicting
$\cK_c=\emptyset$. Choose $\eta$ so that $2\eta<c$.

Define the Lipschitz cutoff function
\[
 \alpha(\lambda)=
 \begin{cases}
  1,&|\lambda-c|\le\eta,\\
  2-|\lambda-c|/\eta,&\eta<|\lambda-c|<2\eta,\\
  0,&|\lambda-c|\ge2\eta.
 \end{cases}
\]
It takes values in $[0,1]$ and has Lipschitz constant $1/\eta$.
Consider the vector field
\begin{equation}\label{eq:flow-field}
 X(x)=-\alpha(f(x))\beta(x)
       \frac{\nabla f(x)}{\max\{\norm{\nabla f(x)}^2,\sigma^2\}}.
\end{equation}
In this step, take $\beta\equiv1$. Let $r_x=\|\nabla f(x)\|$. According to the definitions, $0\le\alpha,\beta\le1$ and
\[
 \|X(x)\|\le\frac{r_x}{\max\{r_x^2,\sigma^2\}}\le\frac1\sigma.
\]
Since the denominator is bounded away from zero, so $X$ is locally Lipschitz. The local existence and uniqueness
of its flow follow by applying the contraction principle to
$x(t)=x(0)+\int_0^tX(x(r))\,\mathrm dr$ on a small time interval.

Let $[0,b_*)$ be the maximal forward interval of $x(t)=\Phi_t(x_0)$ and suppose
$b_*<\infty$. For $0\le s<t<b_*$,
\[
 \|x(t)-x(s)\|
 \le\int_s^t\|X(x(r))\|\,\mathrm dr
 \le\frac{|t-s|}{\sigma}.
\]
Thus for every $\epsilon>0$, the values at any two times in
$(b_*-\sigma\epsilon,b_*)$ differ by less than $\epsilon$. The completeness of $H$ gives a unique
limit $x_*=\lim_{t\to b_*}x(t)$. Extend $x$ continuously by $x(b_*)=x_*$.
The integral equation passes to $t=b_*$, since $X$ is bounded and continuous.
Local existence at $x_*$ supplies a solution for times greater than $b_*$,
and joining it to $x$ attains a solution across $b_*$, with matching derivative
$X(x_*)$. The uniqueness makes it to be an extension of the original trajectory,
contradicting the maximality. So $b_*=\infty$.

The local continuous dependence follows from the same contraction construction
and Gronwall's inequality. Repeating it over finitely many local time
intervals covering any fixed compact trajectory segment proves continuity
of $x_0\mapsto\Phi_t(x_0)$.

For $x(t)=\Phi_t(x)$, the chain rule gives
\[
 \begin{aligned}
 \frac{\mathrm d}{\mathrm dt}f(x(t))
 =Df(x(t))[X(x(t))]=-\alpha(f(x(t)))\beta(x(t))
 \frac{\|\nabla f(x(t))\|^2}
      {\max\{\|\nabla f(x(t))\|^2,\sigma^2\}}.
 \end{aligned}
\]
When $\beta\equiv1$, in the central band
$|f(x(t))-c|\le\eta$, one has $\alpha=1$ and
$\|\nabla f\|\ge\sigma$ by \eqref{eq:empty-gradient}. Therefore, the fraction is exactly $1$, and the derivative equals $-1$.
For any points the bounds are
\[
 -1\le\frac{\mathrm d}{\mathrm dt}f(x(t))\le0.
\]
The derivative is zero where a cutoff vanishes. The unit descent rate is used only on the region where both cutoffs are one and the gradient lower bound holds.

Choose $0<\varepsilon<\min\{\eta,c/4\}$ and $\gamma\in\Gamma$ with
$\max_s f(\gamma(s))\le c+\varepsilon$. Define
$\widetilde\gamma(s)=\Phi_{2\varepsilon}(\gamma(s))$.
The endpoint values are $0$ and a negative number, both below
$c-2\eta>0$. Thus $X=0$ at both endpoints, and uniqueness makes them
to be fixed by the flow. Consequently $\widetilde\gamma\in\Gamma$.

For fixed $s$, let $x_s(t)=\Phi_t(\gamma(s))$, $0\le t\le2\varepsilon$.
If $f(x_s(2\varepsilon))\ge c$, the monotonicity gives
\[
 c\le f(x_s(2\varepsilon))\le f(x_s(t))
       \le f(x_s(0))\le c+\varepsilon<c+\eta.
\]
The entire segment lies in $[c,c+\eta]$, and hence
\[
 f(x_s(2\varepsilon))=f(x_s(0))+\int_0^{2\varepsilon}(-1)\dd=f(x_s(0))-2\varepsilon
       \le c-\varepsilon<c,
\]
a contradiction. Therefore $f(\widetilde\gamma(s))<c$ for every $s$.
The continuous function $f\circ\widetilde\gamma$ attains a maximum
strictly below $c$ on $[0,1]$, contradicting the definition of $c$.
This proves $\cK_c\ne\emptyset$.

Every sequence in $\cK_c$ is a P.S. sequence at $c$, and its
strongly convergent subsequence has limit in $\cK_c$ by continuity.
Thus $\cK_c$ is sequentially compact and hence compact in the Hilbert
norm metric.

\medskip

{\bf Step 2: $\cK_c\cap\overline{\cC}\ne\emptyset$.}

\medskip

Suppose that $\cK_c\cap\overline{\cC}=\emptyset$. The distance between
the compact set $\cK_c$ and the closed set $\overline{\cC}$ is positive. Choose $r>0$ such that
\begin{equation}\label{eq:separation}
 \dist(\cK_c,\overline{\cC})>4r.
\end{equation}
The P.S. condition again gives $\eta,\sigma>0$, with $2\eta<c$,
such that
\begin{equation}\label{eq:gradient-away}
 \|\nabla f(x)\|\ge\sigma
 \quad\text{if }|f(x)-c|\le2\eta
 \text{ and }\dist(x,\cK_c)\ge r.
\end{equation}
Otherwise, it would yield $x_j\in H$ with
\[
 |f(x_j)-c|\le1/j,\qquad \|\nabla f(x_j)\|<1/j,
 \qquad \dist(x_j,\cK_c)\ge r.
\]
A P.S. subsequence converges strongly to $x_*\in\cK_c$,
whereas
\[
 r\le\dist(x_j,\cK_c)\le\|x_j-x_*\|\longrightarrow0
\]
along that subsequence, a contradiction.

Choose $\alpha$ as above, with the present $\eta$, and set
$d(x)=\dist(x,\cK_c)$. Define
\[
 \beta(x)=
 \begin{cases}
  0,&d(x)\le r,\\
  (d(x)-r)/r,&r<d(x)<2r,\\
  1,&d(x)\ge2r.
 \end{cases}
\]
Since $|d(x)-d(y)|\le\|x-y\|$, the function $\beta$ has Lipschitz constant $1/r$. In this step, we use \eqref{eq:flow-field} with these cutoffs and $\sigma$ here. Its local Lipschitz regularity, global flow and speed
bound $1/\sigma$ can be established as in Step 1. If $\alpha(f(x))=\beta(x)=1$, then $|f(x)-c|\le\eta$ and
$d(x)\ge2r$.  Therefore, it follows from \eqref{eq:gradient-away} that $f$ decreases at exactly unit speed.

Choose $0<\varepsilon<\min\{\eta,c/4\}$ small enough such that
\begin{equation}\label{eq:displacement-bound}
 2\varepsilon/\sigma<r.
\end{equation}

Deform an path $\gamma\in\Gamma$ of height at most $c+\varepsilon$ by the flow $\widetilde\gamma(s)=\Phi_{2\varepsilon}(\gamma(s))$ of vector field $X$ for time $2\varepsilon$. Its endpoints will be kept unchanged. We claim
\begin{equation}\label{eq:near-critical-set}
 f(\widetilde\gamma(s))\ge c
 \quad\Longrightarrow\quad
 \dist(\widetilde\gamma(s),\cK_c)<3r.
\end{equation}

Fix $s\in[0,1]$ and write
\[
 x(t)=\Phi_t(\gamma(s)),\quad\forall 0\le t\le2\varepsilon,
 \qquad z=x(2\varepsilon)=\widetilde\gamma(s).
\]
In fact, suppose that $f(z)\ge c$ and $\dist(z,\cK_c)\ge3r$.
The speed bound implies that, for every $t\in[0,2\varepsilon]$, by (\ref{eq:displacement-bound}) it yields
\[
 \|x(t)-z\|\le\int_t^{2\varepsilon}\|X(x(\theta))\|\,\mathrm d\theta
       \le\frac{2\varepsilon-t}{\sigma}
       \le\frac{2\varepsilon}{\sigma}<r.
\]
For every $p\in\cK_c$, the triangle inequality gives
$\|x(t)-p\|\ge\|z-p\|-\|x(t)-z\|$. Taking the infimum over $p$ yields
\[
 \dist(x(t),\cK_c)\ge\dist(z,\cK_c)-\|x(t)-z\|>3r-r=2r.
\]
Therefore $\beta(x(t))=1$. The monotonicity also gives
\[
 c\le f(z)\le f(x(t))\le f(x(0))\le c+\varepsilon<c+\eta,
\]
so $\alpha(f(x(t)))=1$ and \eqref{eq:gradient-away} gives
$\|\nabla f(x(t))\|\ge\sigma$. Therefore, on the entire segment, we have
$\frac{\mathrm d}{\mathrm dt}f(x(t))=-1$. Integrating yields
\[
 f(z)=f(\gamma(s))-2\varepsilon
       \le c+\varepsilon-2\varepsilon=c-\varepsilon<c,
\]
which contradicts to $f(z)\ge c$. Thus the claim (\ref{eq:near-critical-set}) is proved.

Since $\widetilde\gamma\in\Gamma$, its maximum is at least $c$.
As $f(\widetilde\gamma(0))=0<c$, the continuity implies that
\[
 S=\{s\in[0,1]:f(\widetilde\gamma(s))=c\}\ne\emptyset.
\]
It is compact, so $s_0=\min S$ exists, positive and satisfies $f(\widetilde\gamma(s_0))=c$.
For $s<s_0$ one has $f(\widetilde\gamma(s))<c$. For such $s$, we consider the following parametrized deformation
\[
 \eta_s(\theta)=\widetilde\gamma(s\theta),\qquad0\le\theta\le1.
\]
It starts at $0$, ends at $\widetilde\gamma(s)$, and stays in $\{f<c\}$.
Therefore $\widetilde\gamma(s)\in\cC$.
Taking $s_j\uparrow s_0$ gives
$\widetilde\gamma(s_0)\in\overline{\cC}$. Estimate
\eqref{eq:near-critical-set} gives distance less than $3r$ from this point
to $\cK_c$, whereas \eqref{eq:separation} makes that distance greater than
$4r$, a contradiction. Thus $\cK_c\cap\overline{\cC}\ne\emptyset$.
\end{proof}

\begin{corollary}\label{cor:selected-solution}
Assume that $V$ satisfies ($\mathrm{V1'}$) and ($\mathrm{V2}$)-($\mathrm{V5}$). Then, for every $T>0$, there is a nonconstant classical solution $q\in\cS_T$ of \eqref{eq:system} such that
\begin{equation}\label{eq:selected-solution}
 D(\psi_T|_{\cS_T})(q)=0,\qquad
 \psi_T(q)=c_T>0,\qquad q\in\overline{\cC_T}.
\end{equation}
Moreover,
\begin{equation}\label{eq:no-negative-in-component}
 \cC_T\cap\{\psi_T<0\}=\emptyset.
\end{equation}
\end{corollary}

\begin{proof}
Apply Lemma~\ref{lem:selection} to $H=\cS_T$ with its derivative inner product and $f=\psi_T|_{\cS_T}$. Lemmas~\ref{lem:geometry} and~\ref{lem:ps} verify all assumptions and obtain a nonconstant classical solution. If a negative action point belongs to $\cC_T$, a path within this component from zero to that point would have compact image and maximum strictly below $c_T$. It would be an element of $\Gamma_T$, contradicting \eqref{eq:minimax}.
\end{proof}

\subsection{Local geometry of strict sublevels}

For a continuous symmetric bilinear form $B$ on a Hilbert space, its radical is defined by
\[
 \rad(B)=\{h:B(h,w)=0\text{ for every }w\}.
\]
An isotropic vector is one satisfying $B(h,h)=0$. Such a vector need not lie in the radical. This distinction will be important for the cutoff velocity fields.

\begin{lemma}\label{lem:local-exits}
Let $H$ be a real Hilbert space, $f\in C^2(H,\R)$, $Df(x_0)=0$, and $f(x_0)=c>0$. Write $B=D^2f(x_0)$.

If $B(v,v)<0$ for some $v\in H$, let $Y=\ker B(v,\cdot)$.
Then there exist open neighborhoods $\mathcal U\subset H$ of $x_0$ and
$\mathcal V\subset\R\times Y$ of $(0,0)$, and a homeomorphism
$\Theta:\mathcal U\to\mathcal V$ with $\Theta(x_0)=(0,0)$, such that
\begin{equation}\label{eq:concave-form}
 \widetilde f(\xi,y):=f(\Theta^{-1}(\xi,y))
        =c+g(y)-\xi^2,
\end{equation}
where $g$ is a continuous function defined near zero in $Y$ and $g(0)=0$.
There are two fixed local axis points $P_+,P_-$ and an open neighborhood
$O$ of $x_0$ such that every point of $O\cap\{f<c\}$ can be joined to
at least one of $P_+,P_-$ by a path in $\{f<c\}$.

Suppose in addition that $B(h,h)=0$ and $h\notin\rad(B)$. Then every sufficiently nearby point of $\{f<c\}$ can be joined in $\{f<c\}$ to a point of negative value if either of the following conditions holds:
\begin{enumerate}
 \item There are continuous paths $U:[0,s_1]\to H$ and $W:[0,s_2]\to H$, with $s_1,s_2>0$, starting at $x_0$, lying in $\{f<c\}$ for every positive parameter, and ending in $\{f<0\}$, such that
 \[
 U_s=x_0+sh+o(s),\qquad W_s=x_0-dsh+o(s)
 \quad(s\downarrow0)
 \]
 for some $d>0$.
 \item The form $B$ has a two-dimensional negative definite subspace, and there is a path $U$ with the properties above and tangent $h$.
\end{enumerate}
\end{lemma}

\begin{proof}
Since $B(v,v)<0$, for $z\in H$, we define
\[
 p(z)=\frac{B(v,z)}{B(v,v)},\qquad y(z)=z-p(z)v.
\]
Then $B(v,y(z))=0$, so $y(z)\in Y$ and $z=p(z)v+y(z)$.
If $av\in Y$, then $0=aB(v,v)$, hence $a=0$, and the decomposition is
unique. Continuity of $B$ makes $p$ and $y$ to be bounded linear maps, and $Y$
is closed. Thus $H=\R v\oplus Y$ is a topological direct sum.

Define the scalar $C^1$ map
\[
 F(t,y)=Df(x_0+tv+y)[v],\qquad\forall (t,y)\in\R\times Y.
\]
At the origin,
\[
 F(0,0)=0,\qquad\partial_tF(0,0)=B(v,v)<0,
 \qquad D_yF(0,0)[z]=B(z,v)=0\quad(z\in Y).
\]
Therefore, the implicit function theorem gives a $C^1$ function
$t=\varphi(y)$ near $0$ in $Y$ satisfying
$F(\varphi(y),y)=0$ and $\varphi(0)=0$.
Differentiating this identity in an arbitrary direction $z\in Y$ at zero, it yields
\[
 B(v,v)D\varphi(0)[z]+B(z,v)=0.
\]
Thus $D\varphi(0)=0$ since $B(z,v)=0$ for $z\in Y$. In particular, the Fr\'echet differentiability gives
$\varphi(y)=o(\|y\|)$ as $y\to0$ in $Y$.

For $y$ near zero, set
\[
 F_y(t)=f(x_0+(\varphi(y)+t)v+y),\qquad
 g(y)=F_y(0)-c.
\]
The implicit equation gives $F_y'(0)=0$. By continuity of $D^2f$ and
$B(v,v)<0$, there exist $\lambda,\Lambda>0$ and a product neighborhood in
$(t,y)$ on which $\lambda\le-F_y''(t)\le\Lambda$. The Taylor's formula with Cauchy integral remainder, for both positive and negative $t$, gives
\begin{equation}\label{eq:ridge-expansion}
 F_y(t)=c+g(y)-a(t,y)t^2,
\end{equation}
where
\[
 a(t,y)=-\int_0^1(1-\theta)
 D^2f(x_0+(\varphi(y)+\theta t)v+y)[v,v]\,\mathrm d\theta.
\]
The function $a$ is continuous and satisfies
$\lambda/2\le a(t,y)\le\Lambda/2$. Define
\[
 \chi(t,y)=t\sqrt{a(t,y)},\qquad \xi=\chi(t,y).
\]
For fixed $y$, the strict negativity of $F''_y$ and $F_y'(0)=0$ show that $F_y'(t)>0$
for $t<0$ and $F_y'(t)<0$ for $t>0$. So the two branches
$-\sqrt{F_y(0)-F_y(t)}$ for $t<0$ and
$\sqrt{F_y(0)-F_y(t)}$ for $t>0$ are strictly increasing.
They meet at $\chi(0,y)=0$, so $t\mapsto\chi(t,y)$ is strictly
increasing on the whole interval.

For fixed $y$, next we verify the uniform bound on $\partial_t\chi(t,y)$. Here $t$ is a scalar displacement from
$(x_0+\varphi(y)v+y)$ in the direction $v$. For this derivative,
\[
 \begin{aligned}
 \partial_t\chi(0,y)
 &=\lim_{t\to0}\frac{\chi(t,y)-\chi(0,y)}{t}
   =\lim_{t\to0}\sqrt{a(t,y)}=\sqrt{a(0,y)},\\
 a(0,y)
 &=-D^2f(x_0+\varphi(y)v+y)[v,v]
                         \int_0^1(1-\theta)\,\mathrm d\theta\\
 &=-\tfrac12D^2f(x_0+\varphi(y)v+y)[v,v].
 \end{aligned}
\]
This proves the positive value without differentiating the
merely continuous function $a$. The derivative $F_y'(0)$ in the $v$ direction at $(x_0+\varphi(y)v+y)$  is zero.


Let $b(t,y)=-F_y''(t)$, and then $\lambda\le b\le\Lambda$. Note that $\chi(t,y)^2=g(y)-F_y(t)+c$, we have
\[
 F_y'(t)=-t\int_0^1 b(\theta t,y)\,\mathrm d\theta,\quad
 \partial_t\chi(t,y)=
 \frac{\int_0^1 b(\theta t,y)\,\mathrm d\theta}{2\sqrt{a(t,y)}},\quad t\neq 0,
\]
which, together with the value at zero, these formulas imply, throughout a small
coordinate neighborhood,
\[
 0<\frac{\lambda}{2\sqrt{\Lambda/2}}
       \le\partial_t\chi(t,y)
       \le\frac{\Lambda}{2\sqrt{\lambda/2}}<\infty.
\]
The lower bound supplies the uniform Lipschitz bound for the scalar inverse.

Choose $\delta>0$ and $r_*>0$ such
that the preceding estimates hold for $|t|\le\delta$ and
$y\in B_Y(0,r_*)$. Since $a(t,y)\ge \lambda/2>0$, we obtain
\[
 \chi(\delta,y)\ge\delta\sqrt{\lambda/2},\qquad
 \chi(-\delta,y)\le-\delta\sqrt{\lambda/2}.
\]
Fix $0<\rho_*<\delta\sqrt{\lambda/2}$.
Strict monotonicity and the intermediate value theorem give, for each
$(\xi,y)\in(-\rho_*,\rho_*)\times B_Y(0,r_*)$, a unique
$t=\vartheta(\xi,y)\in(-\delta,\delta)$ with $\chi(t,y)=\xi$.

A product neighborhood here is exactly the Cartesian product
\[
 \mathcal V_0=(-\rho_*,\rho_*)\times B_Y(0,r_*).
\]
It is open in $\R\times Y$, and the allowed interval of $\xi$ is the
same for every $y$ in the indicated ball.

Let $m_*>0$ be the uniform lower bound for $\partial_t\chi$.
The mean value theorem gives
\[
 |\vartheta(\xi_1,y)-\vartheta(\xi_2,y)|\le m_*^{-1}|\xi_1-\xi_2|.
\]
If $(\xi_j,y_j)\to(\xi,y)$, set $t_*=\vartheta(\xi,y)$. For all large $j$,
\[
 |\vartheta(\xi_j,y_j)-t_*|
 \le m_*^{-1}|\xi_j-\chi(t_*,y_j)|\longrightarrow0,
\]
by continuity of $\chi$. This proves the continuity of $\vartheta$.
The inverse chart in the Hilbert space is
\[
 \Theta^{-1}(\xi,y)=x_0+(\varphi(y)+\vartheta(\xi,y))v+y.
\]
Its inverse uses the bounded projection $p$, the map
$y=x-x_0-p(x-x_0)v$, and
$\xi=\chi(p(x-x_0)-\varphi(y),y)$, and these are continuous.
Thus it gives a homeomorphic chart on the open neighborhood
$\mathcal U_0=\Theta^{-1}(\mathcal V_0)$ of $x_0$.

Choose $0<2d_0<\rho_*$ and then $0<r_1<r_*$ so small that
$|g(y)|<d_0^2/4$ for $\|y\|<r_1$. Define
\[
 O=\Theta^{-1}\bigl((-d_0/2,d_0/2)\times B_Y(0,r_1)\bigr),
 \qquad P_\pm=\Theta^{-1}(\pm d_0,0).
\]
Then $O$ is an open neighborhood of $x_0$ in the larger coordinate
neighborhood $\mathcal U_0$. The connecting paths below are constructed
in $\mathcal U_0$ and may leave $O$.

Let $x\in O$ have coordinates $(\xi,y)$ and satisfy $g(y)-\xi^2<0$.
Choose $e\in\{1,-1\}$ to agree with the sign of $\xi$; when $\xi=0$
choose either sign, since then $g(y)<0$.
The first path in coordinates is
\[
 ((1-\theta)\xi+\theta e d_0,\ y),\qquad0\le\theta\le1.
\]
Its scalar coordinate has nondecreasing absolute value. Hence its value
of $f-c$ is at most $g(y)-\xi^2<0$.
The second path is
\[
 (e d_0,(1-\theta)y),\qquad0\le\theta\le1.
\]
The ball in $Y$ is convex, so this path has
$f-c=g((1-\theta)y)-d_0^2<d_0^2/4-d_0^2<0$.
Composing both paths with $\Theta^{-1}$, and the concatenation joins $x$ to
$P_e$ within $\{f<c\}$. Since $\varphi(0)=0$ and
$\vartheta(\xi,0)$ has the sign of $\xi$, the points $P_+$ and $P_-$ lie on
opposite $v$ rays in $H$.

Since $h\notin\rad(B)$, choose $w$ with $B(h,w)\ne0$; changing its sign
makes $B(h,w)>0$. Put $v=h-\epsilon w$. For sufficiently small
$\epsilon>0$, isotropy of $h$ gives
\begin{equation}\label{eq:negative-direction}
 \begin{aligned}
 B(v,v)&=-2\epsilon B(h,w)+\epsilon^2B(w,w)<0,\\
 B(v,h)&=-\epsilon B(w,h)\ne0.
 \end{aligned}
\end{equation}
Use the above chart for this $v$. Let $x_s=x_0+sh+r_s$ with
$\|r_s\|=o(s)$, we define
\[
 p_s=\frac{B(v,x_s-x_0)}{B(v,v)},\qquad
 y_s=x_s-x_0-p_sv,\qquad t_s=p_s-\varphi(y_s).
\]
The bounded projections give
\[
 p_s=\frac{B(v,h)}{B(v,v)}s+o(s),\qquad \|y_s\|=O(s).
\]
Since $\varphi(0)=0$ and $D\varphi(0)=0$, one has $\varphi(y_s)=o(s)$. Therefore
$t_s=sB(v,h)/B(v,v)+o(s)$. Also
$a(t_s,y_s)\to a(0,0)=-B(v,v)/2>0$. It follows that
\begin{equation}\label{eq:tangent-coordinate}
 \xi(s)=\chi(t_s,y_s)=t_s\sqrt{a(t_s,y_s)}
 =\sqrt{\frac{-B(v,v)}{2}}\,\frac{B(v,h)}{B(v,v)}s+o(s).
\end{equation}
It follows from \eqref{eq:negative-direction} that the above coefficient is nonzero.
For a path tangent to $-dh$, the coefficient is multiplied by $-d$.

In case (i), (\ref{eq:tangent-coordinate}) shows that the $\xi$ coordinates of $U_s$ and $W_s$ have opposite
signs for all sufficiently small positive parameters. Choose
$s_U,s_W>0$ so small that $U_{s_U},W_{s_W}\in O$ and their scalar
coordinates have opposite signs. The constructed axis paths join
$U_{s_U}$ to one of $P_\pm$ and $W_{s_W}$ to the other.
Starting at an axis, reverse the appropriate connection and follow
$U|_{[s_U,s_1]}$ or $W|_{[s_W,s_2]}$ to its negative value endpoint.
Every parameter used on these paths is positive, so the concatenations
always stay in $\{f<c\}$ and do not pass through $x_0$. So each point of $O\cap\{f<c\}$ connects to a point of negative
value below $c$.

In case (ii), let $N$ be a two-dimensional negative definite subspace. The linear map
$B(v,\cdot): N\to\R$ has a nonzero kernel vector $z$.
Thus $z\in Y$, $B(z,z)<0$, and $B(v,z)=0$.
Let $w_\theta=\cos\theta\,v+\sin\theta\,z$, it yields
\[
 B(w_\theta,w_\theta)
   =\cos^2\theta B(v,v)+\sin^2\theta B(z,z)\le-\kappa,
\]
where $\kappa=\min\{-B(v,v),-B(z,z)\}>0$. The set $\{w_\theta:0\le\theta\le\pi\}$ is compact and bounded.
By the continuity of $D^2f$ at $x_0$ and $Df(x_0)=0$, we obtain
\[
 f(x_0+rw_\theta)=c+\tfrac12r^2B(w_\theta,w_\theta)+o(r^2),
\]
where the remainder is uniform in $\theta$. Therefore, for small $r>0$, the whole curve satisfies $f\le c-\kappa r^2/4<c$ and joins the two $v$ rays. It follow from the strict negativeness of $B(v,v)$ that small nonzero segments of these rays also lie below $c$. Hence $P_+$ and $P_-$ are connected below $c$. A small positive point of $U$ connects to one axis, and its positive-parameter path ends at negative value. Thus both axes, and hence every sufficiently nearby strict sublevel point, connect to negative value below $c$. This proves case (ii) and the lemma.
\end{proof}

\subsection{Negative directions along odd iterates}

The next lemma gives the part of the cutoff Jacobi field construction needed here. It is consistent with the odd iteration estimates of Long~\cite{Long1993}, but the required negative directions are constructed directly.

\begin{lemma}\label{lem:cutoff}
Let $V\in C^2(\R^n,\R)$ be even with $V(0)=0$, and let $q\in\cS_T$ be a nonconstant classical solution of \eqref{eq:system}. Suppose its minimal period is $\tau=T/k$, where $k=2m+1\ge3$. Set $a=\tau/4$ and $L=ka$. On $X_L$, let
\begin{equation}\label{eq:interval-hessian}
 B(u,w)=\int_0^L\bigl(\langle u',w'\rangle
                    -\langle D^2V(q(t))u,w\rangle\bigr)\dd.
\end{equation}
Then $B$ has a negative definite subspace of dimension $m$.
In particular it has a negative two-plane if $k\ge5$. The parameters satisfy
\[
 a=\frac{T}{4k},\qquad L=\frac T4,\qquad
 b=2ma=\frac{mT}{2k}=L-a\in(0,L).
\]
The field
\begin{equation}\label{eq:cutoff-h}
 h(t)=\begin{cases}
       q'(t),&0\le t\le b,\\
       0,&b\le t\le L
      \end{cases}
\end{equation}
belongs to $X_L$, is isotropic, and is not in $\rad(B)$.
\end{lemma}

\begin{proof}
By Lemma~\ref{lem:periods}, it follows $q\in\cS_\tau$. Set $t_j=2ja$ for $0\le j\le m$. Then
\begin{equation}\label{eq:turning-data}
 q(t_j)=(-1)^j q(0),\qquad q'(t_j)=0,\qquad \ddot{q}(t_j)\ne0.
\end{equation}
Indeed, if both velocity and acceleration vanished at a turning point, its position would be a critical point of $V$. The uniqueness would make the solution to be constant. Since $q$ is $C^3$ from Lemma~\ref{lem:ps}, its velocity $y=q'$ satisfies
\begin{equation}\label{eq:jacobi}
 y''+D^2V(q)y=0.
\end{equation}

For $1\le j\le m$, let $h_j$ equal $q'$ on $[t_{j-1},t_j]$ and zero
elsewhere on $[0,L]$. Since $q'(t_{j-1})=q'(t_j)=0$, its values match
at both endpoints and $h_j\in X_L$. Its weak derivative equals
$\ddot{q}$ on $(t_{j-1},t_j)$ and zero elsewhere. Then the integration
by parts and \eqref{eq:jacobi} yield
\[
 \begin{aligned}
 B(h_j,w)
 &=\int_{t_{j-1}}^{t_j}
  \bigl(\langle\ddot{q},w'\rangle-
       \langle D^2V(q)q',w\rangle\bigr)\dd\\
 &=[\langle\ddot{q}(t),w(t)\rangle]_{t_{j-1}}^{t_j}
  -\int_{t_{j-1}}^{t_j}
  \left\langle\frac{\mathrm d}{\mathrm dt}\ddot{q}
            +D^2V(q)q',w\right\rangle\dd\\
 &=[\langle\ddot{q}(t),w(t)\rangle]_{t_{j-1}}^{t_j},\quad \forall w\in X_L.
 \end{aligned}
\]
Taking $w=h_j$ gives zero because $h_j$ vanishes at the two nodes.
For $i\ne j$, the supports of $h_i,h_j$ and their derivatives have
disjoint interiors, so we obtain
\begin{align}
 B(h_i,h_j)&=0\qquad(1\le i,j\le m),\label{eq:isotropic-space}\\
 B(h_j,w)&=\langle\ddot{q}(t_j),w(t_j)\rangle
       -\langle\ddot{q}(t_{j-1}),w(t_{j-1})\rangle.
       \label{eq:boundary-pairings}
\end{align}

We claim that the functionals $\{B(h_j,\cdot)\}_{j=1}^m$ are linearly independent. In fact, suppose $\sum_{j=1}^m\lambda_jB(h_j,w)=0$ for every $w\in X_L$ and some nonzero coefficient. Let $J$ be the largest index with
$\lambda_J\ne0$. Choose a smooth scalar function $\chi$ supported
in a sufficiently small interval about the interior point $t_J$, with
$\chi(t_J)=1$, whose support contains no other node and does not meet $L$.
Set $w(t)=\chi(t)\ddot{q}(t_J)$. Then $w\in X_L$ and
\[
 w(t_J)=\ddot{q}(t_J),\qquad w(t_i)=0,\  \forall i\ne J.
\]
Then the terms with $j<J$ vanish by their endpoint
pairings, and coefficients with $j>J$ are zero. Therefore, formula
\eqref{eq:boundary-pairings} gives \(\lambda_J|\ddot{q}(t_J)|^2=0.\)
This contradicts $\lambda_J\ne0$ and \eqref{eq:turning-data}.

Now we define
\[
 A:X_L\longrightarrow\R^m,\qquad
 A(w)=(B(h_1,w),\ldots,B(h_m,w)),
\]
and let $M=A(X_L)$. If $M$ is a proper subspace of
$\R^m$, there exists a nonzero vector $\lambda=(\lambda_1,\ldots,\lambda_m)\in M^\perp$.
Then the functional \( \ell_\lambda(z)=\sum_{i=1}^m\lambda_i z_i\ (\forall z\in\R^m)\) vanishes on $M$.
In particular, \(\ell_\lambda(A(w))=\sum_{i=1}^m\lambda_iB(h_i,w)=0\) for every $w\in X_L$, which implies that $\{B(h_j,\cdot)\}_{j=1}^m$ are linearly dependent. This contradiction shows that $M=\R^m$ and $A$ is surjective.

For the standard basis $\{e_j\}_{j=1}^m$ of $R^m$, choose $w_j\in X_L$ with $A(w_j)=e_j$; this gives
$B(h_i,w_j)=\delta_{ij}$ exactly. Let $G=(G_{ij})$ with $G_{ij}=B(w_i,w_j)$. Then it follows from
\eqref{eq:isotropic-space} that
\begin{align*}
 B\left(\sum_j\alpha_j(h_j-\epsilon w_j),
        \sum_j\alpha_j(h_j-\epsilon w_j)\right)
   &=-2\epsilon|\alpha|^2+\epsilon^2\alpha^{\mathsf T}G\alpha.
\end{align*}
For sufficiently small $\epsilon>0$, this is strictly negative if $\alpha\ne0$.
Therefore, the vectors $(h_j-\epsilon w_j)$ are linearly independent, and their span is a negative
definite subspace of dimension $m$.

So we have $h=\sum_{j=1}^m h_j\in X_L$. Then it follows from \eqref{eq:isotropic-space} and
\eqref{eq:boundary-pairings} that
\begin{equation}\label{eq:h-pairing}
 \begin{aligned}
 B(h,h)=0, \qquad B(h,w)=\langle\ddot{q}(b),w(b)\rangle
              -\langle\ddot{q}(0),w(0)\rangle.
 \end{aligned}
\end{equation}
Choose $w$ smooth and supported near the interior point $b\in (0,L)$, with $w(b)=\ddot{q}(b)$. Then $B(h,w)=\ddot{q}(b)|^2>0$. This proves $h\notin\rad(B)$.
\end{proof}

\section{Proof of Theorem 1.1}\label{sec:mainproof}

Firstly, Corollary~\ref{cor:selected-solution} gives a classical solution satisfying \eqref{eq:main-selection}. It remains to exclude a shorter minimal period. The exclusion uses only part of the assumptions in Section 1.

\begin{proposition}\label{prop:exit}
Let $V\in C^2(\R^n,\R)$ be even and nonnegative with $V(0)=0$. Suppose $q\in\cS_T$ is a nonconstant classical solution of \eqref{eq:system}, $c=\psi_T(q)>0$, and its minimal period is less than $T$. Then there is an open neighborhood $O$ of $q$ in $\cS_T$ such that every $u\in O$ with $\psi_T(u)<c$ can be joined to a negative action point by a continuous path in $\{\psi_T<c\}$.
\end{proposition}

We prove this proposition in the next three subsections.
By Lemma~\ref{lem:periods}, write
\begin{equation}\label{eq:iteration-parameters}
 \begin{gathered}
 \tau=T/k,\qquad k=2m+1\ge3,\qquad a=\tau/4=\frac{T}{4k},\\
 L=ka=\frac T4,\qquad b=2ma=(k-1)a=L-a.
 \end{gathered}
\end{equation}
Thus $[0,L]$ contains $k$ primitive quarter-periods, and $[b,L]$ is its
last primitive quarter-period.

Regard $q$ as its restriction to $X_L$. It is a critical point of $I_L$ with critical value $c/4$. Set
\begin{equation}\label{eq:turns-energy}
 A_0=q(0),\qquad B_0=q(b)=(-1)^m A_0,
 \qquad E=\frac12|q'(t)|^2+V(q(t)).
\end{equation}

The energy conservation and the symmetries imply
\begin{equation}\label{eq:brake-identities}
 q'(0)=q'(b)=0,\quad q(L)=0,\quad
 q(b+s)=q(b-s),\quad V(A_0)=V(B_0)=E.
\end{equation}

Since $E=|q'(L)|^2/2\ge0$. If $E=0$, then $q(L)=q'(L)=0$, and the uniqueness, together with $\nabla V(0)=0$, gives the zero solution.
So the energy $E$ is positive.

Write
\begin{equation}\label{eq:density}
 \ell(t)=\frac12|q'(t)|^2-V(q(t))=|q'(t)|^2-E,
 \qquad J=\int_0^a\ell(t)\dd.
\end{equation}

Time reflection gives $\ell(-t)=\ell(t)$. Then, by half-antiperiodicity with
primitive half-period $\tau/2=2a$ and the evenness of $V$, we have
\[
 \ell(2a-t)=\ell(t),\qquad
 \int_a^{2a}\ell(t)\dd=\int_0^a\ell(t)\dd=J.
\]
Thus each successive interval of length $a$ has action $J$, including $[b,L]$. Therefore $I_L(q)=kJ$.
Since $V\ge0$, it yields \(\tfrac12|q'(t)|^2=E-V(q(t))\le E\). Integrating over $[0,a]$ gives
\begin{equation}\label{eq:quarter-bound}
 J=\int_0^a|q'(t)|^2\dd-aE\le2aE-aE=aE.
\end{equation}

\subsection{The first action-decreasing surgery}

For $0\le s\le b$, define
\begin{equation}\label{eq:U-path}
 U_s(t)=
 \begin{cases}
  q(t+s),&0\le t\le b-s,\\
  B_0,&b-s\le t\le b,\\
  q(t),&b\le t\le L.
 \end{cases}
\end{equation}
One can easily check that $U_s$ is well-defined, $U_s\in X_L$ and $U_0=q$.

The path $s\mapsto U_s$ is continuous in $H^1$. Indeed, if $s_j\to s$, both $U_{s_j}(t)\to U_s(t)$ and $d_{s_j}(t)\to d_s(t)$
hold except at the finitely many limiting junctions, where $d_s(t)$ denotes the piecewise weak derivative of $U_s(t)$. Note that on
$[0,L]$ they satisfy the uniform bounds
\[
 |U_{s_j}(t)|,|U_s(t)|\le\max_{[0,L]}|q|,
 \qquad |d_{s_j}(t)|,|d_s(t)|\le\max_{[0,L]}|q'|.
\]
Then the dominated convergence theorem gives
\[
 \|U_{s_j}-U_s\|_{H^1}^2
 =\int_0^L|U_{s_j}-U_s|^2\dd+
   \int_0^L|d_{s_j}-d_s|^2\dd\longrightarrow0.
\]

By the definition (\ref{eq:U-path}) of $U_s$, $V(B_0)=E$ and $\ell(t)=|q'(t)|^2-E$, we obtain
 \begin{align}
 I_L(U_s)-I_L(q)
 &=\int_0^{b-s}\ell(t+s)\dd-sV(B_0)+\int_b^L\ell(t)\dd-\int_0^L\ell(t)\dd\notag\\
 &=\int_s^b\ell(r)\,\mathrm dr-sE+\int_b^0\ell(t)\dd=-\int_0^s|q'(t)|^2\dd<0,\quad\forall s>0,
 \label{eq:U-loss}
\end{align}
where the last strict inequality holds because the solution $q(t)$ is nonconstant.

When $s=b$, by \eqref{eq:quarter-bound}, it yields
\begin{equation}\label{eq:U-negative}
 I_L(U_b)=J-bE\le aE-(k-1)aE=-(k-2)aE<0.
\end{equation}

The first-order tangent of this path is the field $h$ in \eqref{eq:cutoff-h}:
\begin{equation}\label{eq:U-tangent}
 U_s=q+sh+o(s)\qquad\text{in }H^1([0,L],\R^n)
 \quad(s\downarrow0).
\end{equation}


Set $R_s\equiv U_s-q-sh$. Let $M=\max_{[0,L]}|\ddot{q}|$ and $\omega(\delta)=\sup\{|\ddot{q}(r)-\ddot{q}(t)|: r,t\in[0,L],\ |r-t|\le\delta\}$.
Obviously, there holds $\omega(\delta)\to0$. Since $q'(b)=0$, the weak derivative of $h$ equals $\ddot{q}$ on $(0,b)$
and $0$ on $(b,L)$. Next we will estimate $R_s$ on three different intervals.

\emph{i) On $[0,b-s]$.} Here $h=q'$, and the fundamental theorem gives
\[
 \begin{aligned}
 R_s(t)&=\int_0^s\bigl(q'(t+r)-q'(t)\bigr)\,\mathrm dr,
                       &|R_s(t)|&\le\tfrac12Ms^2,\\
 R_s'(t)&=\int_0^s\bigl(\ddot{q}(t+r)-\ddot{q}(t)\bigr)\,\mathrm dr,
                       &|R_s'(t)|&\le s\omega(s).
 \end{aligned}
\]
This part contributes at most
$\sqrt b\,(Ms^2/2+s\omega(s))$ to the $H^1$-norm.

\emph{ii) On $[b-s,b]$.} First,
\[
 q'(t)=-\int_t^b\ddot{q}(r)\,\mathrm dr,\qquad
 |q'(t)|\le M(b-t)\le Ms.
\]
Integrating this estimate on $[t,b]$ again gives
\[
 |q(b)-q(t)|\le\int_t^b M(b-r)\,\mathrm dr
              =\tfrac12M(b-t)^2\le\tfrac12Ms^2.
\]
In this interval, almost everywhere there holds
\[
 R_s(t)=q(b)-q(t)-sq'(t),\qquad
 R_s'(t)=-q'(t)-s\ddot{q}(t),
\]
which implies that $|R_s(t)|\le\tfrac32Ms^2$ and $|R_s'(t)|\le2Ms$. Hence
\[
 \|R_s\|_{L^2(b-s,b)}\le\tfrac32Ms^{5/2},\qquad
 \|R_s'\|_{L^2(b-s,b)}\le2Ms^{3/2}.
\]
The total contribution of this part to the $H^1$-norm is $O(s^{3/2})$.

\emph{iii) On $[b,L]$.} One has $U_s=q$ and $h=0$, so $R_s=R_s'=0$
almost everywhere.

Combining the three intervals, for a constant $C$
independent of small $s>0$,
\[
 \|U_s-q-sh\|_{H^1(0,L)}
      \le C\bigl(s^2+s\omega(s)+s^{3/2}\bigr)=o(s).
\]
This establishes \eqref{eq:U-tangent} in the Hilbert norm.

\subsection{The opposite surgery for a triple iterate}

Suppose $k=3$, so $b=2a$ and $L=3a$. For $0\le s\le a$, define
\begin{equation}\label{eq:W-path}
 W_s(t)=
 \begin{cases}
  A_0,&0\le t\le2s,\\
  q(t-2s),&2s\le t\le2a+s,\\
  q(t),&2a+s\le t\le3a.
 \end{cases}
\end{equation}
At the first junction the values match because $q(0)=A_0$. At the second one they match because $q(2a-s)=q(2a+s)$ by the reflection about $2a$. Thus $W_s\in X_{3a}$ and $W_0=q$. The same argument as for $U_s$ proves $H^1$ continuity.

By the definition (\ref{eq:W-path}) of $W_s$ and $V(A_0)=E$, we obtain
\begin{align}
 I_{3a}(W_s)-I_{3a}(q)&=-2sE+\int_0^{2a-s}\ell(r)\,\mathrm dr +\int_{2a+s}^{3a}\ell(t)\dd-\int_{0}^{3a}\ell(t)\dd\notag\\
  &=-2sE-\int_{2a-s}^{2a+s}\ell(t)\dd=-\int_{2a-s}^{2a+s}|q'(t)|^2\dd<0, \quad\forall s>0.
 \label{eq:W-loss}
\end{align}

When $s=a$, it is easy to see that
\begin{equation}\label{eq:W-negative}
 I_{3a}(W_a)=J-2aE\le-aE<0.
\end{equation}

The first-order tangent of this path is the field $-2h$ in \eqref{eq:cutoff-h}:
\begin{equation}\label{eq:W-tangent}
 W_s=q-2sh+o(s)\qquad\text{in }H^1([0,3a],\R^n)
 \quad(s\downarrow0).
\end{equation}

Let $R_s^W\equiv W_s-q+2sh$, $b=2a$, and take $0<s<a/2$.
Use the constant $M$ and $\omega(\delta)$ from the proof for $U$, but being replaced by the interval $[0,3a]$.

On $[2s,b]$, by the fundamental theorem, we obtain
\[
 \begin{aligned}
 R_s^W(t)&=q(t-2s)-q(t)+2sq'(t),
                 &|R_s^W(t)|&\le2Ms^2,\\
 (R_s^W)'(t)&=q'(t-2s)-q'(t)+2s\ddot{q}(t),
                 &|(R_s^W)'(t)|&\le2s\omega(2s).
 \end{aligned}
\]

On $[0,2s]$, one has $q'(0)=0$, $|q'(t)|\le Mt$, and
$|q(t)-q(0)|\le Mt^2/2$. Hence
\[
 R_s^W=q(0)-q+2sq',\qquad (R_s^W)'=-q'+2s\ddot{q},
\]
which establish the bounds $|R_s^W|\le6Ms^2$ and $|(R_s^W)'|\le4Ms$.

On $[b,b+s]$, the cutoff $h$ is zero. Both $t$ and $t-2s$ lie within
$2s$ of $b$, where $q'(b)=0$. Thus
\[
 R_s^W(t)=q(t-2s)-q(t),\qquad
 |R_s^W(t)|\le4Ms^2,\qquad |(R_s^W)'(t)|\le2Ms.
\]

Note that the error is zero for $t\ge b+s$. Since two intervals
$[0,2s]$ and $[b,b+s]$ have total length $3s$, their $H^1$ contribution
is $O(s^{3/2})$. The contribution on $[2s,b]$ is $O(s^2+s\omega(2s))$. Thus
\[
 \|W_s-q+2sh\|_{H^1(0,3a)}
        \le C\bigl(s^2+s\omega(2s)+s^{3/2}\bigr)=o(s),
\]
This establishes \eqref{eq:W-tangent} in the Hilbert norm.

\subsection{Exclusion and completion of the proof}

\begin{proof}[Proof of Proposition~\ref{prop:exit}]
Consider the functional $I_L$ defined on $X_L$ with positive critical value $I_L(q)=c/4$. Its Hessian is the form $B$ in \eqref{eq:interval-hessian}. By Lemma~\ref{lem:cutoff}, the field $h$ is isotropic and not in $\rad(B)$.

When $k=3$, \eqref{eq:U-path}--\eqref{eq:W-tangent} give two paths that lie strictly below the critical value for positive parameter, end at negative values, and have tangents $h$ and $-2h$. By the case~(i) of Lemma~\ref{lem:local-exits} with $d=2$, every sufficiently nearby point of smaller action can be joined, below $I_L(q)$, to a point of negative action.

When $k\ge5$, Lemma~\ref{lem:cutoff} supplies a negative two-plane. By \eqref{eq:U-loss}, \eqref{eq:U-negative} and \eqref{eq:U-tangent}, the path $U$ satisfies all conditions in case~(ii) of Lemma~\ref{lem:local-exits}. So the same conclusion follows.

Finally transport the neighborhood and all connecting paths through the bounded reflection isomorphism \eqref{eq:quarter-space}. The identity $\psi_T=4I_L$ preserves strict inequalities and the sign of the endpoint action. This proves the proposition on $\cS_T$.
\end{proof}

\begin{proof}[Proof of Theorem~\ref{thm:main}]
Fix an integer $n\ge1$, a potential satisfying $(\mathrm{V1'})$ and
$(\mathrm{V2})$--$(\mathrm{V5})$, and $T>0$.
Corollary~\ref{cor:selected-solution} gives a nonconstant classical
solution $q\in\cS_T$ satisfying \eqref{eq:main-selection}.
Let $\tau$ be its minimal positive period, whose existence follows from
Lemma~\ref{lem:periods}.

Suppose $\tau<T$. Proposition~\ref{prop:exit} gives an open neighborhood
$O$ of $q$ with the stated exit property. Choose $\delta>0$ such that
$B_{H^1}(q,\delta)\cap\cS_T\subset O$. Since
$q\in\overline{\cC_T}$, there is
\[
 u\in\cC_T\cap B_{H^1}(q,\delta)\subset O\cap\cC_T.
\]
In particular $\psi_T(u)<c_T$. By the definition of $\cC_T$, there is a
continuous path $\eta_0:[0,1]\to\{\psi_T<c_T\}$ from $0$ to $u$.
Proposition~\ref{prop:exit} gives a continuous path
$\eta_1:[0,1]\to\{\psi_T<c_T\}$ from $u$ to a point $e$ with
$\psi_T(e)<0$. Define
\[
 \gamma(s)=
 \begin{cases}
  \eta_0(2s),&0\le s\le1/2,\\
  \eta_1(2s-1),&1/2\le s\le1.
 \end{cases}
\]
Obviously, $\gamma$ is $H^1$-continuous.
It starts at zero, ends at $e$, and belongs to $\Gamma_T$.
Every point on this path has action less than $c_T$.
The continuous function $\psi_T\circ\gamma$ attains its maximum at some
$s_*\in[0,1]$. Thus, by the definition (\ref{eq:minimax}) of $c_T$, we obtain
\[
 c_T\le\max_{s\in[0,1]}\psi_T(\gamma(s))
     =\psi_T(\gamma(s_*))<c_T,
\]
a contradiction. Therefore $\tau=T$. The time-even and half-antiperiodic
identities follow from $q\in\cS_T$, and the positive action and classical
regularity follow from the selection and compactness lemmas.
\end{proof}

\begin{remark}\label{rem:scope}
The growth outside some ball and the assumptions near the orgin are used for mountain pass geometry and Palais--Smale compactness. Evenness of the potential gives the reflection space and the matching identities in the surgeries. Nonnegativity gives the negative action endpoint estimates \eqref{eq:U-negative} and \eqref{eq:W-negative}. The local argument uses the $C^2$ regularity.
\end{remark}

\medskip

{\bf AI disclosure.} ChatGPT was used to assist with developing the draft arguments, some proof strategies and language refinement. The authors take responsibility for the human verification and exposition.

\end{document}